\documentclass[12pt]{amsart}
\usepackage{latexsym}
\usepackage{amsthm}
\usepackage{amsmath}
\usepackage{amsfonts}
\usepackage{amssymb}
\usepackage{amsrefs}
\usepackage[dvips]{graphicx}
\usepackage{xypic}
\usepackage{xspace}
\input xy
\xyoption{all}

\newtheorem{theo}{Theorem}[section]

\newtheorem{propo}[theo]{Proposition}

\newtheorem{coro}[theo]{Corollary}
\newtheorem{rem}[theo]{Remark}

\theoremstyle{definition}
\newtheorem{defn}[theo]{Defintion}
\newtheorem{notn}[theo]{Notation}

\newcommand\Str{\operatorname{\bf Str}}
\newcommand\Pure{\operatorname{\bf Pure}}
\newcommand\Mono{\operatorname{\bf Mono}}

\newcommand\op{\operatorname{op}}
\newcommand\id{\operatorname{id}}
\newcommand\Id{\operatorname{Id}}

\newcommand\Set{\operatorname{\bf Set}}
\newcommand\Pres{\operatorname{\bf Pres}}
\newcommand\Pos{\operatorname{\bf Pos}}
\newcommand\Ab{\operatorname{\bf Ab}}
\newcommand\Red{\operatorname{\bf Red}}

\newcommand\Mod{\operatorname{\bf Mod}}

\newcommand\ca{\mathcal {A}}
\newcommand\cb{\mathcal {B}}
\newcommand\cc{\mathcal {C}}

\newcommand\cf{\mathcal {F}}
\newcommand\cu{\mathcal {U}}

\newcommand\ck{\mathcal {K}}
\newcommand\cl{\mathcal {L}}

\newcommand\cp{\mathcal {P}}

\newcommand{\sat}{\vDash}
\newcommand{\shlt}{\lhd}
\newcommand{\shleq}{\unlhd}
\newcommand{\shgt}{\rhd}

\newcommand{\restr}{\upharpoonright}
\newcommand{\from}{\colon}
\DeclareMathOperator{\crit}{crit}
\DeclareMathOperator{\Image}{Im}

\newcommand{\al}{\alpha}
\newcommand{\be}{\beta}
\newcommand{\ga}{\gamma}
\newcommand{\de}{\delta}
\newcommand{\ka}{\kappa}
\newcommand{\la}{\lambda}
\newcommand{\calA}{\mathcal{A}}
\newcommand{\calB}{\mathcal{B}}
\newcommand{\calC}{\mathcal{C}}
\newcommand{\calK}{\mathcal{K}}
\newcommand{\calL}{\mathcal{L}}
\newcommand{\calP}{\mathcal{P}}
\DeclareMathOperator{\upset}{\uparrow}
\newcommand{\RssT}{\Red_{\Sigma_1,\Sigma}(T)}

\newcommand{\strcpct}[1]{$L_{#1,\omega}$-compact\xspace}
\newcommand{\abstrcpct}[2]{$#2$-$L_{#1,\omega}$-compact\xspace}
\newcommand{\strcpctness}[1]{$L_{#1,\omega}$-compactness\xspace}
\newcommand{\abstrcpctness}[2]{$#2$-$L_{#1,\omega}$-compactness\xspace}

\date{June 4, 2015}
 
\begin{document}
\title[Accessible images revisited]
{Accessible images revisited}
\author[A. Brooke-Taylor and J. Rosick\'{y}]
{A. Brooke-Taylor$^{*}$ and J. Rosick\'{y}$^{**}$}
\thanks{$^{*}$  Supported by the UK EPSRC Early Career Fellowship EP/K035703/1, ``Bringing set theory and algebraic topology together''.}
\thanks {$^{**}$ Supported by the Grant agency of the Czech republic under the grant P201/12/G028.} 
\address{
\newline A. Brooke-Taylor
\newline School of Mathematics
\newline University of Bristol
\newline Howard House, Queen's Avenue
\newline Bristol, BS8 1SN, UK
\newline a.brooke-taylor@bristol.ac.uk
}
\address{
\newline 
J. Rosick\'y
\newline Department of Mathematics and Statistics, 
\newline Masaryk University, Faculty of Sciences, 
\newline Kotl\'{a}\v{r}sk\'{a} 2, 60000 Brno, Czech Republic
\newline rosicky@math.muni.cz
}
 
\begin{abstract}
We extend and improve the result of Makkai and Par\'e \cite{MP} that
the powerful image of any accessible functor $F$ is accessible, assuming there
exists a sufficiently large strongly compact cardinal.  We reduce the
required large cardinal assumption to the existence of \strcpct{\mu} 
cardinals for sufficiently large $\mu$, and also show that under this assumption
the $\la$-pure powerful image of $F$ is accessible.  
From the first of these statements, 
we obtain that the tameness of every Abstract Elementary Class 
follows from a weaker large cardinal assumption than was previously known.
We provide two ways of
employing the large cardinal assumption to prove
each result --- 
one by a direct ultraproduct construction and one using the
machinery of elementary embeddings of the set-theoretic universe.
\end{abstract} 
\subjclass{ }

\maketitle
\section{Introduction} 
It is well known that the accessibility of the category $\cf$ of free abelian groups depends on set theory --- $\cf$ is accessible 
if there is a strongly compact cardinal and it is not accessible under the axiom of constructibility (see for example \cite{EM}). 
Note that $\cf$ is the image of 
the free abelian group functor $F\from\Set\to\Ab$, where as usual
$\Set$ denotes the
category of sets and $\Ab$ denotes the category of abelian groups.
As a subcategory of $\Ab$, $\cf$ is full, and moreover is closed under
subobjects: a subgroup of a free abelian group is free. 
Generalising from this case,
M.~Makkai and R.~Par\'{e} 
proved that, assuming the existence of arbitrarily large strongly compact cardinals, the powerful image of any accessible functor
is accessible. Here, the \emph{powerful image} of an accessible functor $F\from\ck\to\cl$ is the smallest full subcategory of $\cl$ which contains 
the image of $F$ and is closed under subobjects (see \cite{MP}*{\S5.5}). 
As shown in \cite{LR}, this theorem implies Boney's theorem (see \cite{Bo}) 
asserting that, assuming the existence of arbitrarily large strongly compact cardinals, every abstract elementary class (AEC) is tame. 
A consequence of 
the latter theorem is that, assuming the existence of arbitrarily large strongly compact cardinals, Shelah's Categoricity Conjecture in a successor 
cardinal is true for abstract elementary classes (see \cite{GV}).

The aim of this note is twofold -- firstly we weaken the set theoretic assumption to the existence of arbitrarily large cardinals $\lambda$ admitting 
an \strcpct{\lambda} cardinal (see Definition~\ref{mustrcpct}). 
This thus weakens the set-theoretic assumption known to be sufficient 
to prove Boney's theorem and consequently, 
when paired with a result of Boney and Unger from a forthcoming paper
(see the remarks following Corollary~\ref{allAECtame}),
Shelah's Categoricity Conjecture
in successor cardinals for abstract elementary classes. The second contribution is that, instead of the powerful image, we can use the $\lambda$-pure
powerful image, that is, the closure of the image under $\lambda$-pure subobjects. 
The notion of purity originally arose from model theory, 
where it remains an important concept (see for example
\cite{P} for purity in the model theory of modules), but has also developed to
be a central notion in the general theory of accessible categories 
(see \cite{AR}).

We present two different ways of employing our large cardinal assumption to obtain our result. One follows \cite{MP} directly
and the other uses elementary embeddings of models of set theory 
(see \cite{BB}). 
The first step in each case is to reduce to the case of a suitable reduct
functor $\Red_{\Sigma_1,\Sigma}(T)\to\Str(\Sigma)$ 
where $T$ is a theory in infinitary logic with signature 
$\Sigma_1\supseteq\Sigma$.
It is interesting to 
note that already in 1990 Shelah and Makkai \cite{SM} had obtained 
a categoricity transfer theorem from a successor for $L_{\ka,\omega}$ theories,
where $\ka$ is a strongly compact cardinal; our approach in some sense
brings the general context of AECs back to this infinitary logic setting.

Will Boney and Spencer Unger have independently
obtained similar
results about tameness of AECs from similarly
reduced large cardinal assumptions, 
and moreover can derive large cardinal strength back from tameness assumptions.
Specifically, they show in a forthcoming paper \cite{BU} the equivalence of 
``every AEC $K$ with LS$(K)<\ka$ is $<\!\!\ka$ tame'' with $\ka$ being 
\emph{almost strongly compact}, that is, \strcpct{\mu}
for every $\mu<\ka$.
Note that the existence of an \strcpct{\mu} cardinal for every regular
$\mu$ is equivalent to the existence of a proper class of almost strongly
compact cardinals --- see Proposition~\ref{almstrcpct} below.
We would like to thank Will Boney for discussing his work with us.

\section{Preliminaries}
Recall that a $\lambda$-\emph{accessible} category is a category $\ck$ with $\lambda$-directed colimits, equipped with a set $\ca$ of $\lambda$-presentable
objects such that each object of $\ck$ is a $\lambda$-directed colimit of objects from $\ca$. Here, $\lambda$ is a regular cardinal and an object $K$ is
$\lambda$-presentable if its hom-functor $\ck(K,-)\from\ck\to\Set$ preserves $\lambda$-directed colimits. A category is \emph{accessible} if it is 
$\lambda$-accessible for some regular cardinal $\lambda$. 
See \cite{AR} for an introduction to these categories.

A functor $F\from\ck\to\cl$ is $\lambda$-\emph{accessible} if $\ck$ and $\cl$ are 
$\lambda$-accessible categories and $F$ preserves $\lambda$-directed colimits. 
It is \emph{accessible} if it is $\lambda$-accessible for some regular cardinal $\lambda$. 
For any accessible functor $F$ there are arbitrarily large regular
cardinals $\la$ such that $F$ is $\la$-accessible and preserves 
$\la$-presentable objects --- this is the \emph{Uniformization Theorem}
\cite[2.19]{AR}.
Similarly, a subcategory $\calA$ of a category $\calL$ is 
\emph{accessibly embedded} if it is full and there is some regular cardinal
$\la$ such that $\calA$ is closed under $\la$-directed colimits in $\calL$.

For two regular cardinals $\ka$ and $\ka'$ we say that $\ka$ is 
\emph{sharply less than} $\ka'$, written $\ka\shlt\ka'$, 
if $\ka<\ka'$ and for every $\la<\ka'$, the set 
$[\la]^{<\ka}$ of subsets of $\la$ of cardinality less than $\ka$,
ordered by subset inclusion $\subseteq$, has a cofinal subset of 
cardinality less than $\ka'$.
This rather set-theoretic relation on cardinals is important in the theory
of accessible categories because for $\ka<\ka'$ regular cardinals,
$\ka\shlt\ka'$ if and only if any $\ka$-accessible category 
is $\ka'$-accessible ---
see \cite[Theorem~2.11]{AR} or \cite[Theorems~2.3.10 \& 2.3.14]{MP}.
Note that $\shlt$ is transitive: this can be seen directly
\cite[Proposition~2.3.2]{MP} or 
by appeal to the above accessibility equivalents.
Also note that if $\ka\leq\la$ are regular cardinals then 
$\ka\shlt(\la^{<\ka})^+$
(since $|[\la^{<\ka}]^{<\ka}|=\la^{<\ka}$), 
and the supremum of any set of cardinals sharply
greater than $\ka$ is itself sharply greater than $\ka$.
Thus, for every regular $\ka$ 
there is a closed unbounded class of $\ka'$ such that $\ka\shlt\ka'$.

A morphism $f\from A\to B$ is $\lambda$-\emph{pure} (for $\lambda$ a regular cardinal) provided that in each commutative square
$$
\xymatrix@=3pc{
A \ar[r]^{f} & B \\
X \ar [u]^{u} \ar [r]_{h} &
Y \ar[u]_{v}
}
$$
with $X$ and $Y$ $\lambda$-presentable, $u$ factorizes through $h$, 
that is, $u=th$ for some $t\from Y\to A$.
All needed facts about accessible categories and $\lambda$-pure morphisms can be found in \cite{AR}. In particular, every $\lambda$-pure morphism
in a $\lambda$-accessible category is a monomorphism 
(see \cite[Proposition~2.29]{AR}), and clearly every isomorphism is $\la$-pure
for every $\la$.
It is also easy to see that if $f=g\circ f'$ is $\la$-pure then $f'$ is
$\la$-pure.
In the category $\Str(\Sigma)$ of structures for a $\la$-ary signature
$\Sigma$, $\la$-purity has a natural logical characterisation:
say that an $L_{\la,\la}$ formula is
\emph{positive-primitive} it it is an existentially quantified 
conjunction of atomic formulas. 
Then the $\la$-pure morphisms are precisely the substructure inclusions
that are elementary for positive-primitive formulas of $L_{\la,\la}$
\cite[Proposition~5.34]{AR}.

Following Makkai and Par\'e \cite[Section~5.5]{MP}, we define
the powerful image of a functor as follows.
\begin{defn}\label{powerfulim}
For any functor $F\from\calK\to\calL$, the
\emph{powerful image} $P(F)$ of $F$ is the least full subcategory of $\calL$ 
containing all $FA$, $A\in\calK$,
and closed under subobjects.
\end{defn}

\begin{defn}\label{def2.1}
{
Let $F\from\ck\to\cl$ be a $\lambda$-accessible functor. The $\lambda$-\emph{pure powerful image} $P_\lambda(F)$ of $F$ is the least full
subcategory of $\cl$ containing all $FA$, $A\in\ck$, and closed under $\lambda$-pure subobjects.
}
\end{defn}

Now to the large cardinal axioms we shall employ.

\begin{defn}\label{mustrcpct}
A cardinal $\kappa\geq\mu$ is called 
\emph{\strcpct{\mu}}
if any $\kappa$-complete filter on a set $I$ extends to a $\mu$-complete ultrafilter
on $I$.  
\end{defn}
Thus, if $\kappa$ is a strongly compact cardinal 
(that is, any $\kappa$-complete filter on a set $I$ extends to a $\kappa$-complete ultrafilter on $I$), 
then $\kappa$
is \strcpct{\mu} for all $\mu\leq\kappa$. 
The existence of arbitrarily large strongly compact cardinals 
therefore implies the existence
of \strcpct{\mu} cardinals for all $\mu$.
Moreover note that if $\kappa$ is \strcpct{\mu} and $\ka'\geq\ka$,
then $\ka'$ is also \strcpct{\mu}.

A word is in order about our notation for this large cardinal property,
as there are competing conventions in use.
For the purpose of this discussion, say a cardinal $\ka$ is 
\emph{\abstrcpct{\mu}{\la}}
if any  
$\kappa$-complete filter generated by at most $\la$ many sets 
extends to a $\mu$-complete ultrafilter. 
In recent work
\cites{BB, BM, BM14}, \strcpctness{\mu} has been referred to as
``$\mu$-strong compactness'', mostly with $\mu=\aleph_1$.  
This fits with a tradition in which \abstrcpctness{\ka}{\la} of $\ka$ was
called $\la$-compactness (indeed, this is the terminology of the standard
text \cite{K}), 
with ``strongly'' potentially thought of as indicating
``for all $\la$''.
However, various authors (for example \cites{A,M})
have referred to \abstrcpctness{\ka}{\la} of $\ka$ as 
``$\la$-strong compactness'',
and indeed this fits with the naming paradigm 
for the 
closely related and much more frequently considered $\la$-supercompactness.
Thus, whilst there is no chance of confusion for the $\mu=\aleph_1$ case, 
or indeed if $\mu<\ka$ is explicitly stated,
the terminology ``$\mu$-strongly compact'' could otherwise be problematic.
Our notation itself has a long history (see for example \cites{EA, EM, R}),
and has the benefit of descriptiveness: $\ka$ is $L_{\mu,\omega}$-compact
if and only if for every set $T$ of sentences in the language
$L_{\mu,\omega}$
(or indeed $L_{\mu,\mu}$), if every subset of $T$ of cardinality
less than $\ka$ is satisfiable, then $T$ is satisfiable.
Our notation is perhaps cumbersome when $\la$ is specified, 
and for this general case
Boney and Unger's proposal ``$(\mu,\la)$-strong compactness'' \cite{BU}
might be a better solution,
but since we shall never need to specify $\la$, our 
``\strcpct{\mu}'' seems a more elegant choice than their
``$(\mu,\infty)$-strongly compact''.

A cardinal $\ka$ is said to be \emph{almost strongly compact} if for every
$\mu<\ka$, $\ka$ is \strcpct{\mu}; such cardinals have been used
heavily in the recent work of Boney and Unger on tameness of AECs \cite{BU}.  
At a global level, we have the following equivalence.

\begin{propo}\label{almstrcpct}
There exists a proper class of almost strongly compact cardinals if and only
if for every cardinal $\mu$ there exists an \strcpct{\mu} cardinal.
\end{propo}
\begin{proof}
The forward direction is trivial.  For the converse, suppose that for every
cardinal $\mu$ there exists an \strcpct{\mu} cardinal.  Let $s$ be the
(class) function on cardinals taking each $\mu$ to 
the least \strcpct{\mu} cardinal.  Then as for any class cardinal function,
there is a closed unbounded class $A$ of cardinals $\la$
that are closed under $s$ in the sense that $s(\de)\leq\la$ for all $\de<\la$.
Indeed, for any cardinal $\de$, $\sup_{n\in\omega}s^n(\de)$ is such a cardinal
greater than or equal to $\de$,
and $A$ is clearly closed under taking increasing unions.  But by definition
$A$ is the class of almost strongly compact cardinals.
\end{proof}

One of our proofs will use 
the following equivalent formulation of \strcpctness{\mu} 
due to Bagaria and Magidor.
We use the notation $j``\al$ for
the pointwise image of $j$ on $\al$, that is,
\[
j``\al=\Image(j\restr\al)=\{j(\be):\be\in\al\},
\]
and write $\crit(j)$ for the \emph{critical point} of $j$: the least cardinal
$\delta$ such that $j(\delta)\neq\delta$.

\begin{theo}[{\cite[Theorem~4.7]{BM}}]\label{embedcmpct}
A cardinal
$\ka$ is \strcpct{\mu} if and only if for every $\al\geq\ka$ there
is an elementary embedding $j\from V\to M$ 
definable in $V$, 
where $V$ is the universe of all sets
and $M$ is an inner model of ZFC, such that 
\begin{enumerate}
\item $\ka\geq\crit(j)\geq\mu$
and $M^{\mu}\subset M$,
\item
there is a set $A\supseteq j``\al$ such that $A\in M$ and $M\sat|A|<j(\ka)$.
\end{enumerate}
\end{theo}
Note that $j``\al$ is not assumed to be an element of $M$; 
we shall see that for our
purposes (as in many other cases) it is sufficient to work with
the approximating set $A\supseteq j``\al$ in $M$.

\section{Powerful Images}

\begin{notn}\label{not}
Let $\be$ and $\la$ be cardinals, $\la$ regular.  
We denote by $\ga_{\la,\be}$ the least cardinal greater than or equal to $\be$
such that $\la\shleq\ga_{\la,\be}$.
Following \cite[Examples 2.13(3)]{AR}, we have $\ga_{\la,\be}\leq (2^\beta)^+$
if $\la<\beta$ and $\gamma_{\la,\be}=\lambda$ otherwise. 

Let $\ck$ be a $\lambda$-accessible category. By $\Pres_\lambda\ck$ we denote a full subcategory of $\ck$
such that each $\lambda$-presentable object of $\ck$ is isomorphic to exactly one object of this subcategory.
With $\be=|\Pres_\la\ck|$ we let $\ga_\ck=\ga_{\la,\be}$ and
$\mu_\ck=(\gamma_\ck^{<\gamma_\ck})^+$,
so that $\la\shleq\ga_\calK\shlt\mu_\calK$.
Note that this latter notation is slightly ambiguous 
because $\ck$ is $\lambda$-accessible for many $\lambda$,
but the choice of $\la$ will always be clear.

The category of morphisms of $\ck$ is denoted as $\ck^\to$ because it is the category of functors from the category $\to$ (having two objects
and one non-identity morphism) to $\ck$. 
\end{notn}

\begin{theo}\label{th2.2}
Let $\lambda$ be a regular cardinal and $\calL$ a $\la$-accessible category 
such that there exists an \strcpct{\mu_\cl} cardinal. Then any $\lambda$-pure powerful image 
of a $\lambda$-accessible functor to $\calL$ 
preserving $\mu_\cl$-presentable objects is accessible
and accessibly embedded in $\calL$.
\end{theo}
The requirement on the functor that it preserve $\mu_\calL$-presentable objects
does not materially reduce the applicability of the theorem.
Indeed, if some functor $F$ is $\la'$-accessible, then it follows from
the Uniformization Theorem \cite[2.19]{AR}
that there is $\la\shgt\la'$ such that $F$ is $\la$-accessible
and preserves $\la$-presentable objects.  
With this $\la$, and $\mu_\calL$ chosen as above such that in particular
it is sharply greater than $\la$, 
we have by Remark~2.20 of \cite{AR}
that $F$ also preserves $\mu_\calL$-presentable objects.
\begin{proof}
To assist the reader, we break the proof into three steps, with a
choice in the final step regarding the way in which the large cardinal 
axiom is used. \\

\noindent{\bf Step 1: realise $P_\la(F)$ as a full image.}
Consider a $\lambda$-accessible functor $F\from \ck\to\cl$ which preserves 
$\mu_\cl$-presentable objects. 
To deal with $P_\la(F)$, we shall first recast it as 
the \emph{full image} of a suitable functor $H:\calP\to\calL$, that is,
the full subcategory of $\calL$ with objects of the form $Hp$ for $p$ an
object of $\calP$.
The functor $H$ will be quite natural: it is
the functor from the category $\calP$ of $\la$-pure morphisms
$p:L\to FK$ of $\calL$, taking each such $p$ to its domain $L$.
The work will be in characterising $\calP$ appropriately,
and in particular checking that it and all of the
functors involved are suitably accessible and 
preserve $\mu_\calL$-presentable objects.

Towards this goal, let $\Pure_\lambda(\cl)$ be the subcategory of $\cl$ consisting of all $\cl$-objects and all $\lambda$-pure morphisms. Following \cite[Proposition~2.34]{AR}, this category is accessible, has $\lambda$-directed colimits and the embedding $G\from \Pure_\lambda(\cl)\to\cl$ preserves $\lambda$-directed colimits. Going 
through the proof, one obtains that the category $\Pure_\lambda(\cl)$ is $\mu_\cl$-accessible and the functor $G$ preserves $\mu_\cl$-presentable objects. 
We now very briefly sketch the steps of this argument; 
the reader willing to take the result on
faith may skip ahead to the next paragraph.
First, we may consider the canonical full embedding 
$E\from \cl\to\Set^{\ca^{\op}}$ where $\ca=\Pres_\lambda\cl$;
we use this embedding as a technical device here but note that it will
have an important role to play later in our proof.
Following \cite[Proposition~2.8]{AR}, $E$ preserves $\lambda$-directed colimits and $\lambda$-presentable objects (indeed it sends $\lambda$-presentable objects to finitely presentable ones). Thus it preserves $\gamma_\cl$-directed colimits and $\gamma_\cl$-presentable objects (see \cite[Remarks~2.18(2) and 2.20]{AR}). Hence, following the proof of \cite[Proposition~2.32]{AR}, $\cl$ is closed 
in $\Set^{\ca^{\op}}$ under $\gamma_\cl$-pure subobjects. Finally, following the proof of \cite[Theorem~2.33]{AR}, the category $\Pure_\lambda(\cl)$ 
is $\mu_\cl$-accessible and the functor $G$ preserves $\mu_\cl$-presentable objects. 

Since $\la\shleq\gamma_\cl\triangleleft\mu_\cl$, we have by transitivity that
$\lambda\triangleleft\mu_\cl$, and thus the category $\cl$ is
$\mu_\cl$-accessible. Hence the categories $\Pure_\lambda(\cl)^\to$ and $\cl^\to$ are $\mu_\cl$-accessible (see \cite[Exercise 2.c(1)]{AR}) and the 
induced functor 
$G^\to:\Pure_\la(\cl)^\to\to\cl^\to$ 
preserves $\lambda$-directed colimits and $\mu_\cl$-presentable objects;
$G^\to$ is none other than the inclusion functor.

Now, recall that the objects of the comma category $\Id_\cl\downarrow F$ 
are morphisms $L\to FK$ with $L$ 
in $\cl$ and $K$ in $\ck$. It is shown in \cite[Proposition~2.43]{AR} that $\Id_\cl\downarrow F$ is $\mu_\cl$-accessible, has $\lambda$-directed colimits and the domain and codomain projection functors 
from $\Id_\cl\downarrow F$ to $\cl$ preserve $\lambda$-directed colimits and $\mu_\cl$-presentable objects. Thus the embedding $Q\from \Id_\cl\downarrow F\to\cl^\to$ preserves $\lambda$-directed colimits and $\mu_\cl$-presentable objects. 

We may now appropriately characterise the category of $\la$-pure morphisms
$L\to FK$ of $\calL$: it is the pullback object $\calP$ from 
the pullback diagram
$$
\xymatrix@=3pc{
\Pure_\lambda(\cl)^\to \ar[r]^{G^\to} & \cl^\to \\
\cp \ar [u]^{\bar{Q}} \ar [r]_{\bar{G}} &
\Id_\cl\downarrow F \ar[u]_{Q}.
}
$$
Recall that a functor $F:\calB\to\calC$ is \emph{transportable}
if for every object $B$ of $\calB$ and every isomorphism $c:FB\to C$ of 
$\calC$ there is a unique isomorphism $b:B\to B'$ of $\calB$ such that
$c=Fb$ \cite[page 99]{MP}.
Since isomorphisms are $\lambda$-pure, the functor $G^\to$ is clearly
transportable. 
Thus the pullback above
is a pseudopullback and $\cp$, $\bar{Q}$ and $\bar{G}$ are accessible by \cite[Proposition~5.1.1 and Theorem~5.1.6]{MP}. More precisely, following the Pseudopullback Theorem from \cite{RR} and \cite[Proposition~3.1]{CR}, the category $\cp$ is $\mu_\cl$-accessible, has $\lambda$-directed colimits and the functors $\bar{G},\bar{Q}$ preserve 
$\lambda$-directed colimits and $\mu_\cl$-presentable objects. The composition $G^\to\bar{Q}$ therefore preserves $\lambda$-directed colimits 
and $\mu_\cl$-presentable objects. 

We next consider the domain projection functor $P\from \cl^\to\to\cl$ sending $A\to B$ to $A$. This projection $P$ preserves $\lambda$-directed colimits and 
$\mu_\cl$-presentable objects, so the composition $H=PG^\to\bar{Q}\from \cp\to\cl$ has the same properties. Again, since objects of the category $\cp$ are $\lambda$-pure morphisms $f\from L\to FK$ with $L$ in $\cl$ and $K$ in $\ck$, and $H$ sends $f$ to $L$, the category $P_\lambda(F)$ is equal to the full image of $H$, the full subcategory of $\cl$ consisting of objects $Hf$ with $f$ in $\cp$. \\

\noindent{\bf Step 2: recast in an infinitary language.}
For any small category $\calA$,
the category $\Set^{\ca^{\op}}$ can be considered to be $\Str(\Sigma)$, 
where $\Sigma$ is a many-sorted signature whose sorts are $\ca$-objects and 
whose unary operations with domain sort $A$ and codomain sort $B$ are
morphisms in $\calA$ from $B$ to $A$. 
Taking $\calA$ to be $\Pres_\la(\calL)$ and taking $\Sigma$ correspondingly, 
we have (as mentioned above) a canonical full embedding
$E\from \cl\to\Str(\Sigma)$, which preserves
$\mu_\cl$-directed colimits and $\mu_\cl$-presentable objects. Similarly, we have an embedding $E'\from \cp\to\Str(\Sigma')$ preserving
$\mu_\cl$-directed colimits and $\mu_\cl$-presentable objects where $\Str(\Sigma')=\Set^{\cb^{\op}}$ where $\cb$ is a representative small full
subcategory of $\cp$ of $\mu_\cl$-presentable objects. Following \cite[Corollary~4.18 and Remark~5.33]{AR}, there is an $L_{\mu_\cl,\mu_\cl}(\Sigma')$-theory $T$ such that $\cp$ is equivalent to the category $\Mod(T)$ of models of $T$. In the same way as in the proof of \cite[Theorem~2]{R}, let $\cc=\ca\coprod\cb$
with the corresponding signature $\Sigma_1$ (which is a disjoint union of $\Sigma$ and $\Sigma'$). Then $T$ can be considered as a theory of 
$L_{\mu_\cl,\mu_\cl}(\Sigma_1)$ and $\Str(\Sigma_1)\cong\Str(\Sigma)\times\Str(\Sigma')$. Let $E_1\from \Mod(T)\to\Str(\Sigma_1)$ be induced by $HE$ and $E'$. 
Then the full image of $H$ is equivalent to the full image of the reduct functor $R\from \Mod(T)\to\Str(\Sigma)$. Moreover, $R$ preserves $\mu_\cl$-directed colimits and $\mu_\cl$-presentable objects and the full image of $R$ is the full subcategory $\Red_{\Sigma_1,\Sigma}(T)$ of the category $\Str(\Sigma)$ of $\Sigma$-structures consisting of $\Sigma$-reducts of $T$-models. 

Let $\kappa$ be an \strcpct{\mu_\cl} cardinal. Since any cardinal greater than an \strcpct{\mu_\cl} cardinal is itself
\strcpct{\mu_\cl}, we may assume that $\mu_\cl\shlt\ka$.
Thus $R$ is $\kappa$-accessible and preserves $\kappa$-presentable objects (see \cite{AR} 2.18 and 2.20). Consequently, any object of $\Red_{\Sigma_1,\Sigma}(T)$ is a $\kappa$-directed colimit of $\kappa$-presentable objects 
of $\Str(\Sigma)$ lying in $\Red_{\Sigma_1,\Sigma}(T)$. Thus it remains to prove that $\Red_{\Sigma_1,\Sigma}(T)$ is closed under $\kappa$-directed colimits in $\Str(\Sigma)$.
So
let $D\from I\to\Red_{\Sigma_1,\Sigma}(T)$ be a $\kappa$-directed diagram and $\delta\from D\to C$ be its colimit in $\Str(\Sigma)$; we shall provide two 
proofs that $C$ is in $\Red_{\Sigma_1,\Sigma}(T)$, using the
\strcpctness{\mu_\calL} of $\ka$ in different ways. 
In each case we shall actually show that there is a $\la$-pure morphism from
$C$ to an object of $\Red_{\Sigma_1,\Sigma}(T)$, which clearly suffices. \\

\noindent{\bf Step 3 version (i): a ``hands-on'' ultraproduct.}
This version parallels the original argument of Makkai and Par\'e, 
picking up at the end of page 139 of \cite{MP}.
Let $T_C$ be the \emph{$\lambda$-pure diagram of $C$}, that is,
the set of positive-primitive and negated positive-primitive $L_{\lambda,\lambda}(\Sigma_C)$-formulas valid in $C$, where $\Sigma_C$ is $\Sigma$ augmented with
constant symbols $c_a$ for all of the elements $a$ of $C$. Then $T_C$-models are $\lambda$-pure morphisms
$C\to M$ with $M$ in $\Str(\Sigma)$
(see \cite[Proposition~5.34]{AR}). 
This is analogous to the fact used in \cite{MP} that models of the diagram
$\mathrm{Diag}^{\pm}_C$ of $C$
correspond to one-to-one morphisms with domain $C$.

We hence have that 
$(T\cup T_C)$-models are $\lambda$-pure morphisms $C\to X$ with $X$ in $\Red_{\Sigma_1,\Sigma}(T)$. 
It therefore suffices
to prove that the theory $T\cup T_C$ has a model, 
as the composition of two $\la$-pure
morphisms is $\la$-pure. 
Of course the proof that $T\cup T_C$ has a model 
is by the \strcpctness{\mu} of $\ka$.
Since $I$ is $\kappa$-directed, we can form the $\kappa$-complete filter $\cf$ on $I$ generated 
by the sets $\upset i=\{j\in I\,|\,i\leq j\}$, $i\in I$. 
Since $\kappa$ is \strcpct{\mu_\cl}, 
there is a $\mu_\cl$-complete ultrafilter $\cu$ on $I$ extending
$\cf$. 
We claim that we may take the ultraproduct $\prod_\cu Di$ 
as a model of $T\cup T_C$, with
{\L}o\'s's Theorem applying for our $L_{\mu_\cl,\mu_\cl}$ (and $L_{\la,\la}$)
formulas by the $\mu_\cl$-completeness of $\cu$. 
Indeed, by assumption, $Di$ is a model of $T$ for every $i\in I$. 
For every $i\in I$ and $a\in C$, if there is 
an element of $Di$ that maps to $a$ under the
colimit map to $C$, then interpret the 
constant symbol $c_a$ in $Di$ by 
such an element, and say that $c_a$ has been assigned
\emph{coherently}.  If there is no such element assign the value of
$c_a$ arbitrarily.
Each $Di$ thus becomes a $\Sigma_C$-structure.
Because of the arbitrarily assigned constants $c_a$,
the morphisms of the diagram $D$ may fail
to be $\Sigma_C$-homomorphisms,
but this is irrelevant for our ultraproduct construction.
What is important is that for every positive primitive $L_{\la,\la}$
formula $\varphi$, 
there is an $i$ in $I$ such that every $c_a$ appearing in 
$\varphi$ is assigned coherently in $Di$, and thus also in $Dj$ for every
$j\in\upset i$ --- this follows from the $\ka$-directedness of $D$,
and the fact that, as colimit in $\Str(\Sigma)$, $C$ is simply the direct
limit of the diagram $D$.
Because $\varphi$ is positive primitive 
we then have that, for $j'\geq j\geq i$, if $Dj$ is a model of $\varphi$, 
then $Dj'$ is a model of $\varphi$.
Further, $\varphi$ holds in $C$ if and only if there is some $i'\geq i$
such that $Di'$ satisfies $\varphi$: for $\varphi$ to be true in $C$, 
the existential quantification in
$\varphi$ must be witnessed by particular elements of $C$, 
and there will be some $i'$ large enough that 
that the corresponding constants $c_a$ are assigned coherently in $Di'$
and witness that $\varphi$ holds in $Di'$.
We thus have that if $\varphi$ is true in $C$ then 
it is true in every member of $\upset i'$ for such $i'$, 
and if $\varphi$ is false in $C$ it is false in every member of 
$\upset i$.  Since these sets are in the ultrafilter $\cu$, we have
from {\L}o\'s's Theorem for $L_{\mu_\calL,\mu_\calL}$
that $\prod_\cu Di$ is a model of $\varphi$ if and only if $C$ is,
and may conclude that it is indeed a model of $T\cup T_C$.\\

\noindent{\bf Step 3 version (ii): using an elementary embedding.}
This second approach is via Theorem~\ref{embedcmpct}.
Let $\al=|I|$, and 
let $j\from V\to M$ be an elementary embedding as in 
Theorem~\ref{embedcmpct} for our $\ka$, $\al$ and $\mu=\mu_\calL$.
Note that because $M$ is closed under $<\mu_\calL$-tuples,
$M$ correctly computes whether an object is in $\RssT$:
being a model for the theory $T$ is $\Delta_1$-definable from the
``set of all $<\mu$-tuples'' function $\calP_\mu$, and hence is
absolute between models of set theory that agree on $\calP_\mu$
(see \cite[Proposition~16]{BB}, \cite[Proposition~3.3]{BCMR}).
Consider the diagram $j(D)$ in $M$.  
It is a $j(\ka)$-directed diagram with index category $j(I)$ of cardinality
$j(\al)$.
In particular, we may consider the pointwise image $j``I$ as
a subset of $j(I)$; it has cardinality $\al$, and whilst it
need not be in $M$, by the choice of $j$ as in Theorem~\ref{embedcmpct},
there is a set $A\in M$ such that $A\supseteq j``I$ and 
$M\sat|A|<j(\ka)$.  Hence, in $M$ we may take 
$A\cap j(I)\subseteq j(I)$, and
this set will have cardinality less than $j(\ka)$.  It
therefore has an upper bound $i_0$ in $j(I)$, which in particular is an 
upper bound for $j``I$.

For every object $i$ of $I$, the function $j\restr D(i)\from D(i)\to j(D(i))=
j(D)(j(i))$ 
is a $\Sigma$-homomorphism, by elementarity of $j$.  
Moreover, for every morphism $f\from i\to i'$ of $I$, we have a commuting
square
\[
\xymatrix{
D(i)\ar[r]^{D(f)}\ar[d]_{j\restr D(i)}&D(i')\ar[d]^{j\restr D(i')}\\
j(D(i))\ar[r]_{j(D(f))}&j(D(i'))
}
\]
as one can check by chasing around an element $d\in D(i)$: 
$j(D(f)(d))=j(D(f))(j(d))$, again by elementarity.
Composing these maps with
the maps $j(D(i))=j(D)(j(i))\to j(D)(i_0)$, 
we have a cocone $\nu$ in $\Str{\Sigma}$ (in $V$) from $D$ to
$j(D)(i_0)$, and hence there is a unique homomorphism $h\from C\to j(D)(i_0)$ 
such that $h\circ\delta=\nu$.  Moreover by uniqueness we have that
$j(\delta)(i_0)\circ h$ must equal $j\restr C\from C\to j(C)$, as
$j\restr C\circ \delta(i)=j(\delta)(j(i))\circ j\restr D(i)$
for every object $i$ of $I$ by elementarity.  Since $j\restr C$ is $\la$-pure,
$h$ is $\la$-pure also, and is the desired $\la$-pure morphism to an object
of $\RssT$.
\end{proof}

\begin{rem}\label{re2.3}
{
\em
(1) We have not only proved that $P_\lambda(F)$ is accessible but also
that it is accessibly embedded in $\cl$. 
With the recasting as an inclusion
$\Red_{\Sigma_1,\Sigma}(T)\to\Str(\Sigma)$ as in Steps 1 and 2 of the proof,
this latter aspect also follows from \cite[Theorem~1 and Remark~1(2)]{R}.


(2) Since a $\lambda'$-pure morphism is $\lambda$-pure for $\lambda\leq\lambda'$, we have $P_{\lambda'}(F)\subseteq P_\lambda(F)$ for $\lambda\leq\lambda'$.
Moreover $\cap_\lambda P_\lambda(F)$ is the smallest full subcategory of $\cl$ containing all $FK$, $K\in\ck$, and closed under split subobjects. This
category need not be accessible and accessibly embedded in $\calL$.

To see this, consider the category $\Pos$ of posets and monotone mappings. Let $\ck$ be the category of pairs $(i,p)$ of morphisms of $\Pos$ such that $pi=\id$.
Morphisms are pairs of morphisms $(u,v)\from (i,p)\to (i',p')$ such that $vi=i'u$ and $up=p'v$. 
Let $F\from \ck\to\Pos^\to$ project $(i,p)$ to $i$.
This functor is accessible and its full image consists of split monomorphisms in $\Pos$ (see \cite[Example~3.5(1)]{R1}). 
It is easy to see that this full image is closed under split subobjects. In fact, consider a split monomorphism $(u,v)\from j\to i$ where $i=F(i,p)$. This
means that there is $(r,s)\from i\to j$ such that $ru=\id$ and $sv=\id$. Thus $rpvj=rpiu=ru=\id$.

But the closure of split monomorphisms under $\lambda$-directed colimits precisely consists of $\lambda$-pure monomorphisms (see \cite[Proposition~2.30]{AR}). 
It is easy to see that, for each regular cardinal $\lambda$, there is a $\lambda$-pure monomorphism which does not split. Thus the full image of $F$
is closed under split subobjects but is not 
accessibly embedded into $\Pos$.
}
\end{rem}

\begin{theo}\label{th2.7}
Let $\lambda$ be a regular cardinal and $\calL$ an accessible category 
such that there exists an \strcpct{\mu_\cl} cardinal. 
Then the powerful image of any $\lambda$-accessible functor to $\calL$ preserving $\mu_\cl$-presentable objects is accessible and accessibly embedded in $\calL$.
\end{theo}
\begin{proof}
We proceed as in the proof of \ref{th2.2}, taking the category $\Mono(\cl)$ of $\cl$-objects and monomorphisms instead of $\Pure_\lambda(\cl)$.
This category is closed in $\cl$ under $\lambda$-directed colimits. Since any object of $\cl$ is a $\mu_\cl$-directed colimit of $\lambda$-pure
subobjects $\mu_\cl$-presentable in $\cl$, it is 
certainly a $\mu_\cl$-directed colimit of subobjects $\mu_\cl$-presentable in $\cl$, and so $\Mono(\cl)$ is
$\mu_\cl$-accessible. The functor $E\from \cl\to\Set^{\ca^{\op}}=\Str(\Sigma)$ preserves monomorphisms, and monomorphisms in $\Str(\Sigma)$ are injective
homomorphisms. 
Hence, in the ultraproduct approach to Step 3 of the proof of \ref{th2.2}, 
we may replace $T_C$ by 
the atomic diagram of $C$ consisting of atomic and negated
atomic formulas to obtain the desired result; in the embeddings approach,
we simply use that a right factor of a monomorphism must also be a monomorphism.
\end{proof}

\begin{coro}\label{allmuPowIm}
Suppose for every cardinal $\mu$ there exists an \strcpct{\mu} cardinal.
Then the powerful image of every accessible functor is accessible.
\end{coro}
\begin{proof}
As noted for Theorem~\ref{th2.2} (with reference to \cite[Uniformization Theorem~2.19 and 
Remark~2.20]{AR}), $\la$ and $\mu_\calL$ as in the statement of 
Theorem~\ref{th2.7} may be found for any accessible functor.
\end{proof}

\begin{coro}\label{allAECtame}
Suppose for every cardinal $\mu$ there exists an \strcpct{\mu} cardinal.
Then every AEC is tame.
\end{coro}
\begin{proof}
It was shown in \cite[Theorem~5.2 and Corollary 5.3]{LR} that the accessibility 
of powerful images as proven in Theorem~\ref{th2.7} suffices for tameness of
AECs.  
\end{proof}

Grossberg and VanDieren \cite{GV} showed that assuming amalgamation, joint 
embedding and no maximal models, tameness implies the Shelah Categoricity
Conjecture in successor cardinals for AECs.  
This conjecture was a significant test question for the appropriateness of AECs 
as a framework for generalising first order model theory, 
and indeed the conjecture for arbitrary cardinals
remains an important open question.
Boney's result \cite{Bo}, that if there is a proper class of strongly
compact cardinals then every AEC is tame, was thus a significant breakthrough;
Corollary~\ref{allAECtame} improves upon this important result by
reducing the large cardinal assumption used to prove it.
Moreover, Corollary~\ref{allAECtame} is optimal in this regard:
in a forthcoming paper~\cite{BU} 
(drawing ideas from earlier work of Shelah \cite{S}), 
Boney and Unger show 
that the tameness of all AECs implies that there
is a proper class of almost strongly compact cardinals. 
Thus, with Proposition~\ref{almstrcpct} one has an equivalence between
this large cardinal property, tameness of all AECs, and the accessibility of
powerful images of accessible functors.
Furthermore, Boney and Unger have shown that the extra conditions 
of amalgamation, joint embedding and no maximal models follow from 
this large cardinal axiom, so one indeed has from this assumption 
the Shelah Categoricity Conjecture in successor cardinals for AECs.

\end{document}